\documentclass[11pt]{article}
\usepackage{amscd, amsmath, amssymb, amsthm, graphicx, floatflt, eepic}

\title{Algebraic surfaces and hyperbolic geometry}
\author{Burt Totaro}
\date{  }

\def\Z{\text{\bf Z}}
\def\Q{\text{\bf Q}}
\def\R{\text{\bf R}}
\def\C{\text{\bf C}}
\def\P{\text{\bf P}}

\def\arrow{\rightarrow}
\def\inj{\hookrightarrow}

\def\qed{\ QED }

\def\Aut{\text{Aut}}

\def\PsAut{\text{PsAut}}

\def\Curv{\overline{\text{Curv}}}

\def\Pic{\text{Pic}}
\def\rank{\text{rank}}

\def\im{\text{im}}

\def\dim{\text{dim}}

\def\Nef{\text{Nef}}
\def\Amp{\text{Amp}}
\def\vol{\text{vol}}

\begin{document}
\maketitle

\newtheorem{theorem}{Theorem}[section]
\newtheorem{corollary}[theorem]{Corollary}
\newtheorem{lemma}[theorem]{Lemma}
\newtheorem{conjecture}[theorem]{Conjecture}

\theoremstyle{definition}
\newtheorem{definition}[theorem]{Definition}
\newtheorem{example}[theorem]{Example}

\theoremstyle{remark}
\newtheorem{remark}[theorem]{Remark}

Many properties of a projective algebraic variety
can be encoded by convex cones, such as the ample cone and
the cone of curves. This is especially useful when these cones
have only finitely many edges, as happens for Fano varieties.
For a broader class of varieties which includes Calabi-Yau
varieties and many rationally connected varieties,
the Kawamata-Morrison cone conjecture predicts
the structure of these cones. I like to think of this conjecture
as what comes after the abundance conjecture. Roughly speaking,
the cone theorem of Mori-Kawamata-Shokurov-Koll\'ar-Reid
describes the structure of 
the curves on a projective variety $X$ on which the canonical bundle
$K_X$ has negative degree; the abundance conjecture
would give strong information about the curves on which $K_X$ has degree zero;
and the cone conjecture fully describes the structure
of the curves on which $K_X$ has degree zero.

We give a gentle summary of the proof
of the cone conjecture for algebraic surfaces,
with plenty of examples \cite{Totaro}.
For algebraic surfaces,
these cones are naturally described using hyperbolic geometry,
and the proof can also be formulated in those terms.

Example \ref{non-arithex} shows that the automorphism
group of a K3 surface need not be commensurable with an arithmetic
group. This answers a question
by Barry Mazur \cite[section 7]{Mazur}.

Thanks to John Christian Ottem,
Artie Prendergast-Smith, and Marcus Zibrowius for their
comments.

\section{The main trichotomy}

Let $X$ be a smooth complex projective variety. There are three main types
of varieties. (Not every variety is of one of these three types, but minimal
model theory relates every variety to one of these extreme types.)

{\it Fano. }This means that $-K_X$ is ample. (We recall the definition
of ampleness in section \ref{ample}.)

{\it Calabi-Yau. }We define this to mean that $K_X$ is numerically trivial.

{\it ample canonical bundle. }This means that $K_X$ is ample; it implies
that $X$ is of ``general type.''

Here, for $X$ of complex dimension $n$, the {\it canonical bundle }$K_X$
is the line bundle $\Omega^n_X$ of $n$-forms. We write $-K_X$ for the dual
line bundle $K_X^*$, the determinant of the tangent bundle.

\begin{example}
Let $X$ be a curve, meaning that $X$ has complex dimension 1.
Then $X$ is Fano if it has genus zero, or equivalently if $X$ is isomorphic
to the complex projective line $\P^1$; as a topological space,
this is the 2-sphere.
Next, $X$ is Calabi-Yau if $X$ is an elliptic curve,
meaning that $X$ has genus 1.
Finally, $X$ has ample canonical bundle if it has genus at least 2.
\end{example}

\begin{example}
Let $X$ be a smooth surface in $\P^3$. Then $X$ is Fano if it has degree
at most 3. Next, $X$ is Calabi-Yau if it has degree 4; this is one class
of {\it K3 surfaces}. Finally, $X$ has ample canonical bundle if it has degree
at least 5.
\end{example}

Belonging to one of these three classes of varieties is
equivalent to the existence of
a K\"ahler metric with Ricci curvature of a given sign, by Yau \cite{Yau}.
Precisely,
a smooth projective variety is Fano if and only if
it has a K\"ahler metric with positive
Ricci curvature; it is Calabi-Yau if and only if
it has a Ricci-flat K\"ahler metric;
and it has ample canonical bundle if and only if it has a K\"ahler metric
with negative Ricci curvature.

We think of Fano varieties as the most special class
of varieties, with projective space as a basic example.
Strong support for this idea is provided by
Koll\'ar-Miyaoka-Mori's theorem that smooth Fano varieties of dimension $n$
form a bounded family \cite{KMM}. In particular, there are only finitely many
diffeomorphism
types of smooth Fano varieties of a given dimension.

\begin{example}
Every smooth Fano surface is isomorphic to $\P^1\times \P^1$ or 
to a blow-up of $\P^2$ at at most 8 points. The classification of smooth
Fano 3-folds is also known, by Iskovskikh, Mori, and Mukai;
there are 104 deformation classes \cite{IP}.
\end{example}

By contrast, varieties with ample canonical bundle form a vast and
uncontrollable class. Even in dimension 1, there are infinitely many
topological types of varieties with ample canonical bundle
(curves of genus at least 2). Calabi-Yau varieties are on the border
in terms of complexity. It is a notorious open question whether there
are only finitely many topological types of smooth Calabi-Yau varieties of
a given dimension. This is true in dimension at most 2. In particular, a 
smooth Calabi-Yau
surface is either an abelian surface, a K3 surface, or a quotient of one
of these surfaces by a free action of a finite group (and only finitely many
finite groups occur this way).

\section{Ample line bundles and the cone theorem}
\label{ample}

After a quick review of ample line bundles, this section states
the cone theorem and its application to Fano varieties.
Lazarsfeld's
book is an excellent reference on ample line bundles \cite{Lazarsfeld}.

\begin{definition}
A line bundle $L$ on a projective variety $X$ is {\it ample }if
some positive multiple $nL$ (meaning the line bundle $L^{\otimes n}$) has
enough global sections to give a projective embedding
$$X\inj \P^N.$$
(Here $N=\dim_{\C}H^0(X,nL)-1$.)
\end{definition}

One reason to investigate which line bundles are ample
is in order to classify algebraic varieties. For classification,
it is essential to know how to describe a variety with given
properties as a subvariety of a certain projective
space defined by equations of certain degrees.

\begin{example}
For $X$ a curve, $L$ is ample on $X$
if and only if it has positive degree. We write
$L\cdot X  = \deg(L|_X)\in \Z$.
\end{example}

An {\it $\R$-divisor }on a smooth projective variety $X$ is a finite sum
$$D=\sum a_i D_i$$
with $a_i\in \R$ and each $D_i$
an irreducible divisor (codimension-one subvariety) in $X$.
Write $N^1(X)$ for the ``N\'eron-Severi'' 
real vector space of $\R$-divisors modulo {\it numerical
equivalence}: $D_1\equiv D_2$ if $D_1\cdot C=D_2\cdot C$ for all curves $C$
in $X$. (For me, a {\it curve }is irreducible.)

We can also define $N^1(X)$ as the subspace of
the cohomology $H^2(X,\R)$ spanned by divisors.
In particular, it is a finite-dimensional real vector space.
The dual vector space $N_1(X)$
is the space of 1-cycles $\sum a_iC_i$ modulo numerical equivalence,
where $C_i$ are curves on $X$. We can identify $N_1(X)$ with the subspace
of the homology $H_2(X,\R)$ spanned by algebraic curves.

\begin{definition}
The {\it closed cone of curves }$\Curv(X)$ is the closed convex cone
in $N_1(X)$ spanned by curves on $X$.
\end{definition}

\begin{definition}
An $\R$-divisor $D$ is {\it nef }if $D\cdot C\geq 0$ for all curves $C$ in $X$.
Likewise, a line bundle $L$ on $X$
is {\it nef }if the class $[L]$ of $L$ (also called
the first Chern class $c_1(L)$) in $N^1(X)$ is nef. That is,
$L$ has nonnegative
degree on all curves in $X$.
\end{definition}

Thus $\Nef(X)\subset N^1(X)$ is a closed convex cone, the {\it dual cone }to
$\Curv(X)\subset N_1(X)$.

\begin{theorem}
(Kleiman) A line bundle $L$ is ample if and only if $[L]$ is in the interior
of the nef cone in $N^1(X)$.
\end{theorem}

This is a {\it numerical }characterization of ampleness. It shows that we know
the ample cone $\Amp(X)\subset N^1(X)$ if we know the cone of curves
$\Curv(X)\subset N_1(X)$. The following theorem gives a good
understanding of the ``$K$-negative'' half of the cone of curves
\cite[Theorem 3.7]{KMbook}. A {\it rational curve }means a curve
that is birational to $\P^1$.

\begin{theorem}
(Cone theorem; Mori, Shokurov, Kawamata, Reid, Koll\'ar). Let $X$ be a smooth
projective variety. Write $K_X^{<0}=\{ u\in N_1(X): K_X\cdot u<0\}$.
Then every extremal ray of $\Curv(X)\cap K_X^{<0}$ is isolated, 
spanned by a rational curve, and can be contracted.
\end{theorem}

In particular, every extremal ray of $\Curv(X)\cap K_X^{<0}$ is
rational (meaning that it is spanned by a $\Q$-linear combination
of curves, not just an $\R$-linear combination), since it is spanned
by a single curve. A {\it contraction }of a normal projective
variety $X$ means a surjection from $X$ onto a normal
projective variety $Y$ with connected fibers. A contraction is determined
by a face of the cone of curves $\Curv(X)$, the set of elements of $\Curv(X)$
whose image under the pushforward map $N_1(X)\arrow N_1(Y)$
is zero. The last statement in the cone
theorem means that every extremal ray in the $K$-negative half-space
corresponds to a contraction of $X$.

\begin{corollary}
For a Fano variety $X$, the cone of curves $\Curv(X)$ (and therefore
the dual cone $\Nef(X)$) is rational polyhedral.
\end{corollary}

A rational polyhedral cone means the closed convex cone
spanned by finitely many rational points.

\begin{proof}
Since $-K_X$ is ample, $K_X$ is negative on
all of $\Curv(X)-\{0\}$. So the cone theorem applies to all
the extremal rays of $\Curv(X)$. Since they are isolated and
live in a compact space (the unit sphere), $\Curv(X)$ has only
finitely many extremal rays. The cone theorem also gives
that these rays are rational.
\end{proof}

It follows, in particular, that a Fano variety has only finitely
many different contractions. A simple example is the blow-up $X$
of $\P^2$ at one point, which is Fano. In this case,
$\Curv(X)$ is a closed strongly convex cone
in the two-dimensional real vector space $N_1(X)$, and so it has exactly
two 1-dimensional faces. We can write down two contractions of $X$,
$X\arrow \P^2$ (contracting a $(-1)$-curve) and $X\arrow \P^1$
(expressing $X$ as a $\P^1$-bundle over $\P^1$). Each of these morphisms must
contract one of the two 1-dimensional faces of $\Curv(X)$. Because the cone
has no other nontrivial faces, these are the only nontrivial
contractions of $X$.

\section{Beyond Fano varieties}

``Just beyond'' Fano varieties, the cone of curves and the nef cone need not be
rational polyhedral. Lazarsfeld's book \cite{Lazarsfeld}
gives many examples of this type, as do other books on minimal model theory
\cite{Debarre,KMbook}.

\subsection{Example}
Let $X$ be the blow-up of $\P^2$ at $n$
very general points. For $n\leq 8$, $X$ is Fano, and so
$\Curv(X)$ is rational polyhedral. In more detail, for
$2\leq n\leq 8$, $\Curv(X)$ is the convex cone spanned
by the finitely many $(-1)$-curves in $X$. (A $(-1)$-curve
on a surface $X$ means a curve $C$ isomorphic to $\P^1$
with self-intersection number $C^2=-1$.) For example,
when $n=6$, $X$ can be identified with a cubic surface,
and the $(-1)$-curves are the famous 27 lines on $X$.

But for $n\geq 9$, $X$ is not Fano, since $(-K_X)^2=9-n$
(whereas a projective variety has positive degree
with respect to any ample line bundle).
For $p_1,\ldots,p_n$ very general points in $\P^2$, $X$ 
contains infinitely many $(-1)$-curves; see Hartshorne
\cite[Exercise V.4.15]{Hartshorne}. Every curve
$C$ with $C^2<0$ on a surface spans an isolated extremal
ray of $\Curv(X)$, and so $\Curv(X)$ is not rational
polyhedral. 

Notice that a $(-1)$-curve $C$ has $K_X\cdot C=-1$, and so these
infinitely many isolated extremal rays are on the ``good'' ($K$-negative) side
of the cone of curves, in the sense of the cone theorem. 
The $K$-positive side is a mystery. It is conjectured (Harbourne-Hirschowitz) 
that the closed cone
of curves of a very general blow-up of $\P^2$ at $n\geq 10$ points is 
the closed convex cone spanned by the $(-1)$-curves and the 
``round'' positive cone $\{x\in N_1(X): x^2\geq 0 \text{ and }H\cdot x
\geq 0\}$,
where $H$ is a fixed ample line bundle. This includes the famous
Nagata conjecture \cite[Remark 5.1.14]{Lazarsfeld}
as a special case. By de Fernex,
even if the Harbourne-Hirschowitz conjecture
is correct, the intersection of $\Curv(X)$ with the
$K$-positive half-space, for $X$ a very general blow-up
of $\P^2$ at $n\geq 11$ points, is bigger than
the intersection of the positive cone with the $K$-positive half-space,
because the $(-1)$-curves stick out a lot from the positive cone
\cite{defernex}.
 
\subsection{Example}

Calabi-Yau varieties (varieties with $K_X\equiv 0$) are also ``just beyond''
Fano varieties
($-K_X$ ample). Again, the cone of curves of a Calabi-Yau variety need
not be rational polyhedral.

For example, let $X$ be an abelian surface, so $X\cong \C^2/\Lambda$ for some
lattice $\Lambda\cong \Z^4$ such that $X$ is projective.
Then $\Curv(X)=\Nef(X)$
is a round cone, the positive cone
$$\{ x\in N^1(X): x^2\geq 0\text{ and }H\cdot x\geq 0\},$$
where $H$ is a fixed ample line bundle.
(Divisors and 1-cycles are the same thing on a surface,
and so the cones $\Curv(X)$ and $\Nef(X)$ lie in the same vector space
$N^1(X)$.)
Thus the nef cone is not rational polyhedral if $X$
has Picard number $\rho(X):=
\dim_{\R}N^1(X)$ at least 3 (and sometimes when $\rho=2$).

For a K3 surface, the closed cone of curves may be round,
or may be the closed cone spanned by the $(-2)$-curves
in $X$. (One of those two properties must hold, by Kov\'acs \cite{Kovacs}.)
There may be finitely or
infinitely many $(-2)$-curves. See section \ref{coneexample}
for an example.

\section{The cone conjecture}
\label{coneconjecture}

But there is a good substitute for the cone theorem for Calabi-Yau varieties,
the {\it Morrison-Kawamata cone conjecture}. In dimension 2, this is a theorem,
by Sterk-Looijenga-Namikawa \cite{Sterk, Namikawa, Kawamata}. We call
this Sterk's theorem for convenience:

\begin{theorem}
Let $X$ be a smooth complex projective Calabi-Yau surface
(meaning that $K_X$ is numerically trivial). Then the action
of the automorphism
group $\Aut(X)$ on the nef cone $\Nef(X)\subset N^1(X)$ has a rational
polyhedral fundamental domain.
\end{theorem}

\begin{remark}
For any variety $X$, if $\Nef(X)$ is rational polyhedral,
then the group $\Aut^*(X):=\im(\Aut(X)\arrow GL(N^1(X)))$ is finite.
This is easy: the group $\Aut^*(X)$ must permute the set consisting of the smallest integral
point on each extremal ray of $\Nef(X)$. Sterk's theorem implies the remarkable
statement that
the converse is also true for Calabi-Yau surfaces. That is, if the cone
$\Nef(X)$ is not rational polyhedral, then $\Aut^*(X)$ must be infinite.
Note that $\Aut^*(X)$ coincides with the discrete part of the automorphism
group of $X$ up to finite groups, because $\ker(\Aut(X)\arrow GL(N^1(X)))$
is an algebraic group and hence has only finitely many connected
components.
\end{remark}

Sterk's theorem should generalize to Calabi-Yau varieties of any dimension
(the Morrison-Kawamata cone conjecture). But in dimension 2,
we can visualize it better, using hyperbolic geometry.

Indeed, let $X$ be any smooth projective surface. The intersection form
on $N^1(X)$ always has signature $(1,n)$ for some $n$ (the Hodge
index theorem). So $\{ x\in N^1(X): x^2>0\}$ has two connected
components, and the positive cone
$\{x\in N^1(X): x^2>0\text{ and }H\cdot x>0\}$
is the standard round cone. As a result, we can identify
the quotient of the positive cone by $\R^{>0}$
with {\it hyperbolic $n$-space}. One way to see this is that the negative of
the Lorentzian metric on $N^1(X)=\R^{1,n}$ restricted to the quadric
$\{x^2=1\}$ is a Riemannian metric with curvature $-1$.

For any projective surface $X$, $\Aut(X)$ preserves the intersection
form on $N^1(X)$. So $\Aut^*(X)$ is always a group of isometries
of hyperbolic $n$-space, where $n=\rho(X)-1$.

By definition, two groups $G_1$ and $G_2$ are {\it commensurable},
written $G_1\doteq G_2$, if some finite-index subgroup of $G_1$
is isomorphic to a finite-index subgroup of $G_2$. Since the groups
we consider are all virtually torsion-free, we are free to replace
a group $G$ by $G/N$ for a finite normal subgroup $N$ (that is,
$G$ and $G/N$ are commensurable).

\subsection{Examples}
\label{coneexample}

For an abelian surface $X$ with Picard number at least 3, the cone
$\Nef(X)$ is round, and so $\Aut^*(X)$ must be infinite by Sterk's theorem.
(For abelian surfaces, the possible automorphism
groups were known long before \cite[section 21]{MumfordAV}.)

\begin{floatingfigure}[r]{0.3\textwidth}
\centering
\includegraphics[scale=0.15]{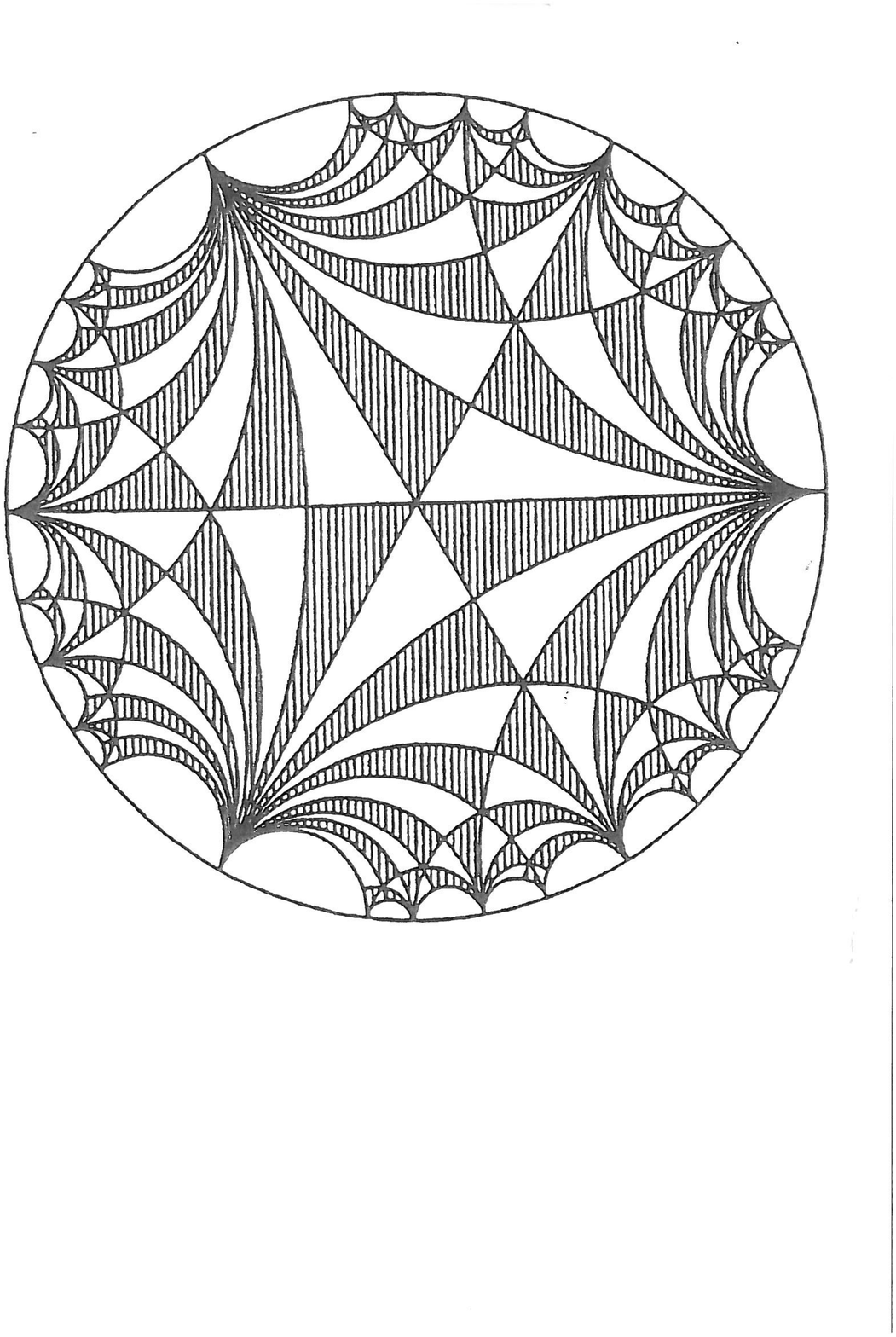}
\end{floatingfigure}
For example, let $X=E\times E$ with $E$ an elliptic curve (not having
complex multiplication). Then $\rho(X)=3$, with $N^1(X)$ spanned by
the curves $E\times 0$, $0 \times E$, and the diagonal
$\Delta_E$ in $E\times E$. So $\Aut^*(X)$ must be infinite.
In fact,
$$\Aut^*(X)\cong PGL(2,\Z).$$
Here $GL(2,\Z)$ acts
on $E\times E$ as on the direct sum of any abelian group with itself.
This agrees with Sterk's theorem, which says that $\Aut^*(X)$ acts on the
hyperbolic plane with a rational polyhedral fundamental domain;
a fundamental domain for $PGL(2,\Z)$
acting on the hyperbolic plane (not preserving orientation) is given
by any of the triangles in the figure.

For a K3 surface, the cone $\Nef(X)$ may or may not be the whole
positive cone. For any projective surface, the nef cone
modulo scalars is a convex subset of hyperbolic space. A finite
polytope in hyperbolic space (even if some vertices are at infinity)
has finite volume. So Sterk's theorem implies that, for a Calabi-Yau
surface, $\Aut^*(X)$ acts with {\it finite covolume }on the convex
set $\Nef(X)/\R^{>0}$ in hyperbolic space.

\begin{floatingfigure}[r]{0.3\textwidth}
\centering
\includegraphics[scale=0.3]{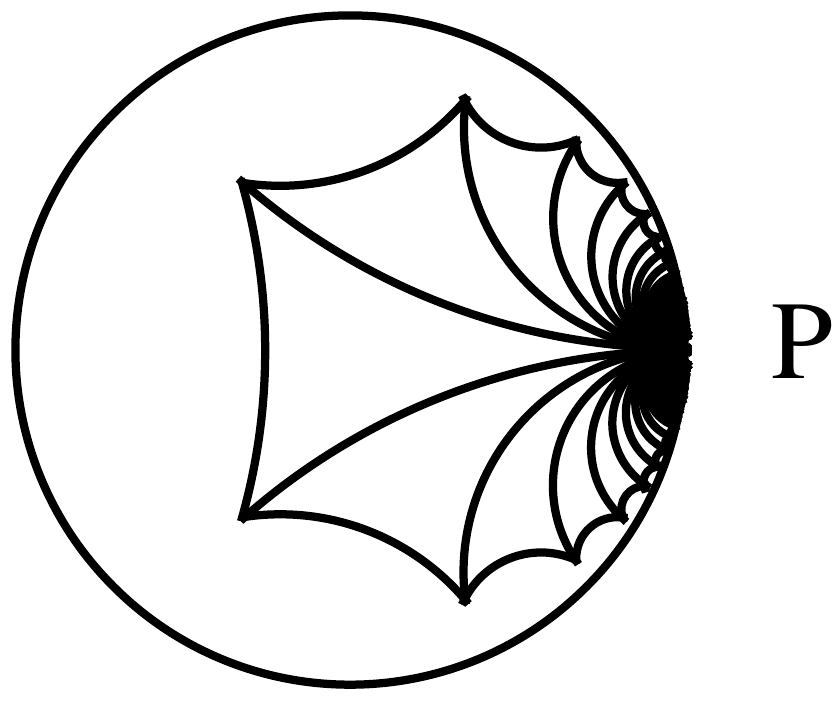}
\end{floatingfigure}
For example, let $X$ be a K3 surface such that $\Pic(X)$ is isomorphic
to $\Z^3$
with intersection form
$$\begin{pmatrix}
0&1&1\\
1& -2 &0\\
1&0& -2
\end{pmatrix}.$$
Such a surface exists, since Nikulin showed that 
every even lattice of rank at most 10 with
signature $(1,*)$ is the Picard lattice of some complex
projective K3 surface \cite[section 1, part 12]{Nikulinint}.
Using the ideas of section
\ref{sterk}, one computes that the nef cone of $X$ modulo scalars is the convex
subset of the hyperbolic plane shown in the figure. The surface $X$
has a unique elliptic fibration $X\arrow \P^1$,
given by a nef line bundle $P$
with $\langle P,P\rangle =0$. The line bundle $P$ appears in the figure
as the point where $\Nef(X)/\R^{>0}$ meets the circle at infinity.
And $X$ contains infinitely many $(-2)$-curves, whose orthogonal complements
are the codimension-1 faces of the nef cone. Sterk's theorem says
that $\Aut(X)$ must act on the nef cone with rational polyhedral fundamental
domain. In this example, one computes that
$\Aut(X)$ is commensurable with the Mordell-Weil
group of the elliptic fibration ($\Pic^0$ of the generic fiber
of $X\arrow \P^1$), which is isomorphic to $\Z$. One also finds 
that all the $(-2)$-curves in $X$
are sections of the elliptic fibration. The Mordell-Weil group moves
one section to any other section, and so it divides the nef cone 
into rational polyhedral cones as in the figure.

\section{Outline of the proof of Sterk's theorem}
\label{sterk}

We discuss the proof of Sterk's theorem for K3 surfaces.
The proof for abelian surfaces is the same, but simpler
(since an abelian surface contains no $(-2)$-curves),
and these cases imply the case of quotients of K3 surfaces
or abelian surfaces by a finite group. For details,
see Kawamata \cite{Kawamata},
based on the earlier papers \cite{Sterk, Namikawa}.

The proof of Sterk's theorem for K3 surfaces relies on the Torelli theorem
of Piatetski-Shapiro and Shafarevich.
That is, any isomorphism of Hodge structures between
two K3s is realized by an isomorphism of K3s if it
maps the nef cone into the nef cone. In particular,
this lets us construct automorphisms of a K3 surface
$X$: up to finite index,
every element of the integral orthogonal group $O(\Pic(X))$
that preserves the cone $\Nef(X)$ is realized by an
automorphism of $X$. (Here $\Pic(X)\cong \Z^{\rho}$, and the intersection
form has signature $(1,\rho(X)-1)$ on $\Pic(X)$.)

Moreover, $\Nef(X)/\R^{>0}$ is a very special convex set in hyperbolic space
$H_{\rho-1}$: it is the closure of a Weyl chamber
for a discrete reflection group $W$ acting on $H_{\rho-1}$. We can define
$W$ as the group generated by all reflections in vectors $x\in \Pic(X)$
with $x^2=-2$, or (what turns out to be the same) the group generated
by reflections in all $(-2)$-curves in $X$. By the first description, $W$
is a {\it normal }subgroup of $O(\Pic(X))$. In fact, up to finite groups,
$O(\Pic(X))$ is the semidirect product group
$$O(\Pic(X))\doteq \Aut(X)\ltimes W.$$

By general results on arithmetic groups going back to Minkowski,
$O(\Pic(X))$ acts on the positive cone in $N^1(X)$ with a rational
polyhedral fundamental domain $D$.  (This fundamental domain is not
at all unique.) And the reflection group $W$ acts on the positive cone
with fundamental domain the nef cone of $X$. Therefore, after we arrange
for $D$ to be contained in the nef cone,
$\Aut(X)$ must act on the nef cone
with the same rational polyhedral fundamental domain $D$, up to finite index.
Sterk's theorem is proved.

\section{Non-arithmetic automorphism groups}

In this section, we show for the first time that
the discrete part of the automorphism group of a smooth
projective variety need not be commensurable with an arithmetic group.
This answers a question raised by Mazur \cite[section 7]{Mazur}.
Corollary \ref{non-arithK3} applies to a large class of K3 surfaces.

An {\it arithmetic group }is
a subgroup of some $\Q$-algebraic group $H_{\Q}$ which is commensurable
with $H(\Z)$ for some integral structure on $H_{\Q}$; this condition
is independent of the integral structure \cite{Serre}. We view
arithmetic groups as abstract groups, not as subgroups of a fixed
Lie group.

Borcherds gave an example of a K3 surface whose automorphism
group is not isomorphic to an arithmetic group
\cite[Example 5.8]{Borcherds}. But, as he says, the automorphism group
in his example has a nonabelian free subgroup of finite index, and so
it is commensurable with the arithmetic group $SL(2,\Z)$.
Examples of K3 surfaces with explicit generators
of the automorphism group
have been given by Keum, Kondo, Vinberg, and others;
see Dolgachev \cite[section 5]{DolgachevReflection} for a survey. 

Although they need not be commensurable with arithmetic groups,
the automorphism groups $G$ of K3 surfaces 
are very well-behaved in terms of geometric group theory.
More generally this is true for the discrete part $G$ of
the automorphism group of a surface
$X$ which can be given the structure of a klt Calabi-Yau pair, as defined
in section \ref{klt}.
Namely, $G$ acts cocompactly on a CAT(0) space
(a precise notion of a metric space with nonpositive curvature).
Indeed, the nef cone modulo scalars is a closed convex subset of
hyperbolic space, and thus a CAT($-1$) space
\cite[Example II.1.15]{BridsonH}.
Removing a $G$-invariant set of disjoint open horoballs gives a
CAT(0) space on which $G$ acts properly and cocompactly, by the proof of
\cite[Theorem II.11.27]{BridsonH}. This implies all the finiteness
properties one could want, even though $G$ need not be arithmetic.
In particular: $G$
is finitely presented,
a finite-index subgroup of $G$ has a finite CW complex as classifying space,
and $G$ has only finitely many conjugacy classes of finite subgroups
\cite[Theorem III.$\Gamma$.1.1]{BridsonH}.

For smooth projective varieties in general,
very little is known. For example,
is the discrete part $G$ of the automorphism group always finitely
generated? The question is open even for smooth projective rational
surfaces. About the only thing one can say for an arbitrary
smooth projective variety $X$ is that
$G$ modulo a finite group injects into $GL(\rho(X),\Z)$,
by the comments in section \ref{coneconjecture}.

In Theorem \ref{lattice},
a {\it lattice }means a finitely generated free abelian group
with a symmetric bilinear form that is nondegenerate $\otimes\Q$.

\begin{theorem}
\label{lattice}
Let $M$ be a lattice of signature $(1,n)$ for $n\geq 3$. Let
$G$ be a subgroup of infinite index in $O(M)$. Suppose
that $G$ contains $\Z^{n-1}$ as a subgroup of infinite index.
Then $G$ is not commensurable with an arithmetic group.
\end{theorem}

\begin{corollary}
\label{non-arithK3}
Let $X$ be a K3 surface over any field, with Picard number
at least 4. Suppose that $X$ has an elliptic fibration with
no reducible fibers and a second elliptic fibration with
Mordell-Weil rank positive. (For example, the latter property holds if the
second fibration also has no reducible fibers.)
Suppose also that $X$ contains a $(-2)$-curve.
Then the automorphism group of $X$ is a discrete group that is not
commensurable with an arithmetic group.
\end{corollary}

\begin{example}
\label{non-arithex}
Let $X$ be the double cover of $\P^1\times \P^1
=\{ ([x,y],[u,v])\}$ ramified along the following curve of degree
$(4,4)$:
\begin{multline*}
0 = 16x^4u^4+xy^3u^4+y^4u^3v-40x^4u^2v^2-x^3yu^2v^2-x^2y^2uv^3+33x^4v^4
-10x^2y^2v^4+y^4v^4.
\end{multline*}
Then $X$ is a K3 surface whose automorphism group (over $\Q$,
or over $\overline{\Q}$) is not commensurable
with an arithmetic group.
\end{example}

\begin{proof}[Proof of Theorem \ref{lattice}]
We can view $O(M)$
as a discrete
group of isometries of hyperbolic $n$-space. Every solvable subgroup of
$O(M)$ is virtually abelian 
\cite[Corollary II.11.28 and Theorem III.$\Gamma$.1.1]{BridsonH}.
By the classification of isometries of hyperbolic
space as elliptic, parabolic, or hyperbolic \cite{Vinberg}, the centralizer
of any subgroup $\Z\subset O(M)$ is either commensurable with $\Z$
(if a generator $g$ of $\Z$ is hyperbolic) or commensurable with
$\Z^a$ for some $a\leq n-1$
(if $g$ is parabolic). These properties pass to the subgroup $G$ of $O(M)$.
Also, $G$ is not virtually abelian, because
it contains $\Z^{n-1}$ as a subgroup of infinite index, and $\Z^{n-1}$
is the largest abelian subgroup of $O(M)$ up to finite index.
Finally, $G$ acts properly and not cocompactly on hyperbolic $n$-space,
and so $G$ has virtual cohomological dimension at most $n-1$.

Suppose that $G$ is commensurable
with some arithmetic group $\Gamma$. 
Thus $\Gamma$ is a subgroup of some $\Q$-algebraic group
$H_{\Q}$, and $\Gamma$ is commensurable with $H(\Z)$ for some integral
structure on $H_{\Q}$. We freely change $\Gamma$ by finite
groups in what follows. So we can assume that $H_{\Q}$ is connected.
After replacing $H_{\Q}$ by the kernel
of some homomorphism to a product of copies of the multiplicative
group $G_m$ over $\Q$, we can assume that $\Gamma $ is a {\it lattice }in
the real group $H(\R)$ (meaning that $\vol(H(\R)/\Gamma)<\infty$),
by Borel and Harish-Chandra \cite{BorelH}.

Every connected $\Q$-algebraic group $H_{\Q}$ is a semidirect product
$R_{\Q}\ltimes U_{\Q}$ where $R_{\Q}$ is reductive and $U_{\Q}$ is unipotent
\cite[Theorem 5.1]{BorelS}.
By the independence of the choice of integral structure, we can assume
that $\Gamma= R(\Z)\ltimes U(\Z)$ for some arithmetic subgroups
$R(\Z)$ of $R_{\Q}$ and $U(\Z)$ of $U_{\Q}$. Since every solvable
subgroup of $G$ is virtually abelian, $U_{\Q}$ is abelian,
and $U(\Z)\cong \Z^a$ for
some $a$. The conjugation action of $R_{\Q}$ on $U_{\Q}$ must be trivial;
otherwise $\Gamma$ would contain a solvable group of the form
$\Z\ltimes \Z^a$ which is not virtually abelian. Thus $\Gamma=
\Z^a\times R_{\Z}$. But the properties of centralizers in $G$
imply that $G$ does not contain the product of $\Z$ with any infinite
nonabelian group. Therefore, $a=0$ and $H_{\Q}$ is reductive.

Modulo finite groups, the reductive group $H_{\Q}$ is a product
of $\Q$-simple groups and tori, and $\Gamma$ is a corresponding product
modulo finite groups. Since $G$ does not contain the product of $\Z$
with any infinite nonabelian group, $H_{\Q}$ must be $\Q$-simple.
Since the lattice $\Gamma$ in $H(\R)$ is isomorphic to the discrete subgroup
$G$ of $O(M)\subset O(n,1)$ (after passing to finite-index subgroups),
Prasad showed that $\dim(H(\R)/K_H)\leq
\dim(O(n,1)/O(n))=n$, where $K_H$ is a maximal compact subgroup
of $H(\R)$. Moreover, since $G$ has infinite index in $O(M)$ 
and hence infinite covolume in $O(n,1)$,
Prasad showed that either $\dim(H(\R)/K_H)\leq n-1$ or else
$\dim(H(\R)/K_H)=n$ and there is a homomorphism from $H(\R)$ onto
$PSL(2,\R)$ \cite[Theorem B]{Prasad}.

Suppose that $\dim(H(\R)/K_H)\leq n-1$. We know that $\Gamma$ acts properly
on $H(\R)/K_H$ and that $\Gamma$ contains $\Z^{n-1}$. The quotient
$\Z^{n-1}\backslash H(\R)/K_H$ is a manifold of dimension
$n-1$ with the homotopy type of the $(n-1)$-torus (in particular,
with nonzero cohomology in dimension $n-1$), and so it must
be compact. So $\Z^{n-1}$ has finite index in $\Gamma$, contradicting
our assumption.

So $\dim(H(\R)/K_H)=n$ and $H(\R)$ maps onto $PSL(2,\R)$.
We can assume that $H_{\Q}$ is simply connected. Since
$H$ is $\Q$-simple, $H$ is equal to the restriction of scalars
$R_{K/\Q}L$ for some number field $K$ and some absolutely simple
and simply connected
group $L$ over $K$ \cite[section 3.1]{Tits}. Since $H(\R)$ maps onto
$PSL(2,\R)$, $L$ must be a form of $SL(2)$.
We know that $G\cong \Gamma$
has virtual cohomological dimension at most $n-1$, and so $\Gamma$
must be a non-cocompact subgroup of $H(\R)$. Equivalently, $H$
has $\Q$-rank greater than zero \cite{BorelH},
and so $\rank_K(L)=\rank_{\Q}(H)$ 
is greater than zero.
Therefore, $L$ is isomorphic to $SL(2)$ over $K$.

It follows that
$\Gamma$ is commensurable with $SL(2,o_K)$,
where $o_K$ is the ring of integers of $K$. So we can assume that
$\Gamma$ contains
the semidirect product
$$o_K^*\ltimes o_K=\bigg\{ \begin{pmatrix} a & b\\
0 & 1/a \end{pmatrix} \bigg\}\subset SL(2,o_K).$$
 If the group of units $o_K^*$
has positive rank, then $o_K^*\ltimes o_K$
is a solvable group which is not virtually abelian. So the group of units
is finite, which means that $K$ is either $\Q$ or an imaginary
quadratic field, by Dirichlet. If $K$ is imaginary quadratic,
then $H_{\Q}=R_{K/\Q}SL(2)$ and $H(\R)=SL(2,\C)$, which does not map
onto $PSL(2,\R)$. Therefore $K=\Q$ and $H_{\Q}=SL(2)$. It follows
that $\Gamma$ is commensurable with $SL(2,\Z)$. So
$\Gamma$ is commensurable
with a free group. This contradicts
that $G\cong\Gamma$ contains $\Z^{n-1}$ with $n\geq 3$.
\end{proof}

\begin{proof}[Proof of Corollary \ref{non-arithK3}]
Let $M$ be the Picard
lattice of $X$, that is, $M=\Pic(X)$ with the intersection form.
Then $M$ has signature $(1,n)$ by the Hodge index theorem, where
$n\geq 3$ since $X$ has Picard number at least 4.

For an elliptic fibration $X\arrow \P^1$
with no reducible fibers,
the Mordell-Weil group of the fibration has rank $\rho(X)-2=n-1$
by the Shioda-Tate formula \cite[Cor.\ 1.5]{Shioda},
which is easy to check in this case. So the first elliptic fibration
of $X$ gives an inclusion of $\Z^{n-1}$ into $G=\Aut^*(X)$.
The second elliptic fibration gives an inclusion of $\Z^a$
into $G$ for some $a>0$.
In the action of $G$ on hyperbolic $n$-space, the Mordell-Weil
group of each elliptic fibration is a group of parabolic transformations
fixing the point at infinity that corresponds
to the class $e\in M$ of a fiber (which has $\langle e,e\rangle =0$).
Since a parabolic transformation
fixes only one point of the sphere at infinity, the subgroups $\Z^{n-1}$
and $\Z^a$ in $G$ intersect only in the identity. It follows that
the subgroup $\Z^{n-1}$ has infinite index
in $G$.

We are given that $X$ contains a $(-2)$-curve $C$. I claim that
$C$ has infinitely many translates
under the Mordell-Weil group $\Z^{n-1}$. Indeed, any curve with finitely many
orbits under $\Z^{n-1}$ must be contained in a fiber of $X\arrow \P^1$.
Since all fibers are irreducible, the fibers have
self-intersection 0, not $-2$. Thus $X$ contains infinitely many $(-2)$-curves.
Therefore the group $W\subset O(M)$ generated by reflections in $(-2)$-vectors
is infinite. Here $W$ acts simply transitively on the Weyl chambers of
the positive cone (separated by hyperplanes $v^{\perp}$ with $v$ a
$(-2)$-vector), whereas $G=\Aut^*(X)$ preserves one Weyl chamber,
the ample cone of $X$.
So $G$ and $W$ intersect only in the identity. Since $W$
is infinite, $G$ has infinite index in $O(M)$.
By Theorem \ref{lattice},
$G$ is not commensurable
with an arithmetic group.
\end{proof}

\begin{proof}[Proof of Example \ref{non-arithex}]
The given curve $C$ 
in the linear system $|O(4,4)|={|-2K_{\P^1\times \P^1}|}$
is smooth. One can check this with Macaulay 2, for example.
Therefore the double cover $X$ of $\P^1\times \P^1$ ramified along $C$
is a smooth K3 surface. The two projections from $X$ to $\P^1$
are elliptic fibrations. Typically, such a double cover
$\pi:X\arrow \P^1\times \P^1$ would have
Picard number 2, but the curve $C$ has been chosen to be
tangent at 4 points to each of two curves of degree $(1,1)$,
$D_1=\{xv=yu\}$ and $D_2=\{xv=-yu\}$. (These points
are $[x,y]=[u,v]$ equal to $[1,1],[1,2],[1,-1],[1,-2]$ on $D_1$ 
and $[x,y]=[u,-v]$ equal to $[1,1],[1,2],[1,-1],[1,-2]$ on $D_2$.)
It follows that the double
covering is trivial over $D_1$ and $D_2$, outside the ramification
curve $C$: the inverse image
in $X$ of each curve $D_i$ is a union of two
curves, $\pi^{-1}(D_i)=E_i\cup F_i$, meeting
transversely at 4 points. The smooth rational curves $E_1,F_1,E_2,F_2$
on $X$ are $(-2)$-curves, since $X$ is a K3 surface.

The curves $D_1$ and $D_2$ meet transversely at the two points
$[x,y]=[u,v]$ equal to $[1,0]$ or $[0,1]$. 
Let us compute that the double covering $\pi:X\arrow \P^1\times \P^1$
is trivial over the union
of $D_1$ and $D_2$ (outside the ramification curve $C$).
Indeed, if we write $X$ as
$w^2=f(x,y,z,w)$ where $f$ is the given polynomial of degree $(4,4)$,
then a section of $\pi $ over $D_1\cup D_2$
is given by $w=4x^2u^2-5x^2v^2+y^2v^2$. We can name the curves
$E_i,F_i$ so that the image of this section is $E_1\cup E_2$
and the image of the section $w=-(4x^2u^2-5x^2v^2+y^2v^2)$
is $F_1\cup F_2$. Then $E_1$ and $F_2$ are disjoint. So
the intersection
form among the divisors $\pi^*O(1,0),\pi^*O(0,1), E_1, F_2$ on $X$ is given by
$$\begin{pmatrix}
0&2&1&1\\
2&0&1&1\\
1&1& -2 &0\\
1&1&0& -2
\end{pmatrix}$$
Since this matrix has determinant $-32$, not zero, $X$ has Picard number
at least 4.

Finally, we compute that the two projections from $C\subset \P^1\times
\P^1$ to $\P^1$ are each ramified over 24 distinct points in $\P^1$.
It follows that all fibers of our two elliptic fibrations $X\arrow \P^1$
are irreducible. By Corollary \ref{non-arithK3}, the automorphism
group of $X$ (over $\C$, or equivalently over $\overline{\Q}$)
is not commensurable with an arithmetic group. Our calculations
have all worked over $\Q$, and so Corollary \ref{non-arithK3}
also gives that $\Aut(X_{\Q})$ is not commensurable
with an arithmetic group.
\end{proof}

\section{Klt pairs}
\label{klt}

We will see that the previous results can be generalized from Calabi-Yau
varieties to a broader class of varieties using the language of pairs.
For the rest of the paper, we work over the complex numbers.

A normal variety $X$ is {\it $\Q$-factorial }if for every point
$p$ and every codimension-one subvariety $S$ through $p$, there is a regular
function on some neighborhood of $p$ that vanishes exactly on $S$
(to some positive order).

\begin{definition}
A pair $(X,\Delta)$ is a $\Q$-factorial projective variety $X$
with an effective $\R$-divisor $\Delta$ on $X$. 
\end{definition}

Notice that $\Delta$ is an actual $\R$-divisor $\Delta=\sum a_i\Delta_i$,
not a numerical equivalence class of divisors.
We think of $K_X+\Delta$ as the canonical bundle of the pair
$(X,\Delta)$. The following definition picks out an important
class of ``mildly singular'' pairs.

\begin{definition}
A pair $(X,\Delta)$ is {\it klt }(Kawamata log terminal)
if the following holds. Let $\pi:\widetilde{X}
\arrow X$ be a resolution of singularities. Suppose that the union of the
exceptional set of $\pi$ (the subset of $\widetilde{X}$ where $\pi$
is not an isomorphism) with $\pi^{-1}(\Delta)$ is a divisor
with simple normal crossings. Define a divisor $\widetilde{\Delta}$
on $\widetilde{X}$ by
$$K_{\widetilde{X}}+\widetilde{\Delta}=\pi^*(K_X+\Delta).$$
We say that $(X,\Delta)$ is klt if all coefficients
of $\widetilde{\Delta}$ are less than 1. This property is independent
of the choice of resolution.
\end{definition}

\begin{example}
A surface $X=(X,0)$ is klt if and only if $X$ has only quotient singularities
\cite[Proposition 4.18]{KMbook}.
\end{example}

\begin{example}
For a smooth variety $X$ and $\Delta$ a divisor with simple normal crossings
(and some coefficients), the pair $(X,\Delta)$ is klt if and only if
$\Delta$ has coefficients less than 1.
\end{example}

All the main results of minimal model theory, such as the cone theorem,
generalize from smooth varieties to klt pairs. For example,
the Fano case of the cone theorem becomes \cite[Theorem 3.7]{KMbook}:

\begin{theorem}
Let $(X,\Delta)$ be a klt Fano pair, meaning that $-(K_X+\Delta)$
is ample. Then $\Curv(X)$ (and hence the dual cone $\Nef(X)$)
is rational polyhedral.
\end{theorem}

Notice that the conclusion does not involve the divisor $\Delta$.
This shows the power of the language of pairs. A variety $X$ may not
be Fano, but if we can find an $\R$-divisor $\Delta$ that makes
$(X,\Delta)$ a klt Fano pair, then we get the same conclusion
(that the cone of curves and the nef cone are rational polyhedral)
as if $X$ were Fano.

\begin{example}
Let $X$ be the blow-up of $\P^2$ at any number of points on a smooth conic.
As an exercise, the reader can
write down an $\R$-divisor $\Delta$ such that
$(X,\Delta)$ is a klt Fano pair. This proves that the nef cone of $X$
is rational polyhedral, as Galindo-Monserrat \cite[Corollary 3.3]{GalM},
Mukai \cite{Mukaifin}, and Castravet-Tevelev \cite{CT} proved
by other methods. These surfaces are definitely not Fano if we blow
up 6 or more points. Their Betti numbers are unbounded,
in contrast to the smooth Fano surfaces. 

More generally, Testa, V\'arilly-Alvarado,
and Velasco proved that every smooth projective rational surface $X$
with $-K_X$ big has finitely generated Cox ring \cite{TVV}.
Finite generation of the Cox ring (the ring of all sections
of all line bundles)
is stronger than the nef cone being rational polyhedral,
by the analysis of Hu and Keel \cite{HK}.
Chenyang Xu showed that a rational surface with $-K_X$ big
need not have any divisor $\Delta$ with $(X,\Delta)$
a klt Fano pair \cite{TVV}. I do not know whether the blow-ups of
higher-dimensional projective spaces considered by Mukai
and Castravet-Tevelev have a divisor $\Delta$ with $(X,\Delta)$
a klt Fano pair \cite{Mukaifin, CT}.
\end{example}

It is therefore natural to extend the Morrison-Kawamata
cone conjecture from Calabi-Yau varieties to {\it Calabi-Yau
pairs }$(X,\Delta)$, meaning that $K_X+\Delta\equiv 0$.
The conjecture is reasonable, since we can prove it in
dimension 2 \cite{Totaro}.

\begin{theorem}
\label{pair}
Let $(X,\Delta)$ be a klt Calabi-Yau pair of dimension 2.
Then $\Aut(X,\Delta)$ (and also $\Aut(X)$) acts with a rational
polyhedral fundamental domain on the cone $\Nef(X)\subset N^1(X)$.
\end{theorem}

Here is a more concrete consequence
of Theorem \ref{pair}:

\begin{corollary} 
\cite{Totaro} Let $(X,\Delta)$ be a klt Calabi-Yau pair
of dimension 2. Then there are only finitely many contractions
of $X$ {\it up to automorphisms of $X$}. Related to that:
$\Aut(X)$ has {\it finitely many orbits }on the set of curves
in $X$ with negative self-intersection.
\end{corollary}

This was shown in one class of examples by Dolgachev-Zhang
\cite{DZ}. These results are false for surfaces in general, even for some
smooth rational surfaces:

\begin{example}
Let $X$ be the blow-up of $\P^2$ at 9 very general points.
Then $\Nef(X)$ is not rational polyhedral, since $X$ contains
infinitely many $(-1)$-curves. But $\Aut(X)=1$
\cite[Proposition 8]{Gizatullin}, and so
the conclusion fails for $X$.

Moreover, let $\Delta$ be the unique smooth cubic curve in $\P^2$
through the 9 points, with coefficient 1. Then $-K_X\equiv \Delta$,
and so $(X,\Delta)$ is a {\it log-canonical }(and even canonical)
Calabi-Yau pair. The theorem therefore fails for such pairs.
\end{example}

We now give a classical example (besides the case $\Delta=0$ of Calabi-Yau
surfaces)
where Theorem \ref{pair} applies.

\subsection{Example}
Let $X$ be the blow-up of $\P^2$ at 9 points $p_1,\ldots,p_9$ which are the
intersection of two cubic curves. Then taking linear combinations
of  the two cubics gives a $\P^1$-family of elliptic curves through
the 9 points. These curves become disjoint on the blow-up $X$,
and so we have an elliptic fibration $X\arrow \P^1$. This morphism
is given by the linear system $|-K_X|$. Using that, we see that the
$(-1)$-curves on $X$ are exactly the  sections of the elliptic
fibration $X\arrow \P^1$.

In most cases, the Mordell-Weil group
of $X\arrow \P^1$ is $\doteq\Z^8$. So $X$ contains infinitely
many $(-1)$-curves, and so the cone $\Nef(X)$ is not rational
polyhedral. But $\Aut(X)$ acts transitively on the set
of $(-1)$-curves, by translations using the group
structure on the fibers of $X\arrow \P^1$. That leads
to the proof of Theorem \ref{pair} in this example.
(The theorem applies, in the sense that there is an
$\R$-divisor $\Delta$ with $(X,\Delta)$ klt Calabi-Yau:
let $\Delta$ be the sum of two smooth fibers of $X\arrow \P^1$
with coefficients $1/2$, for example.)

\section{The cone conjecture in dimension greater than 2}

In higher dimensions, the cone conjecture also predicts that
a klt Calabi-Yau pair $(X,\Delta)$ has only finitely
many small $\Q$-factorial modifications $X\dashrightarrow X_1$
{\it up to pseudo-automorphisms of $X$}. (See Kawamata \cite{Kawamata}
and \cite{Totaro} for the full statement of the cone conjecture
in higher dimensions.)
A pseudo-automorphism means a birational automorphism which
is an isomorphism in codimension 1.

More generally, the conjecture implies that $X$ has only finitely
many birational contractions $X\dashrightarrow Y$ modulo
pseudo-automorphisms of $X$, where a birational contraction
means a dominant rational map that extracts no divisors.
There can be infinitely
many small modifications if we do not divide out by 
the group $\PsAut(X)$ of pseudo-automorphisms of $X$.

Kawamata proved a relative version of the cone conjecture for
a 3-fold $X$ with a K3 fibration or elliptic fibration $X\arrow S$
\cite{Kawamata}.
Here $X$ can have infinitely many minimal models (or
small modifications) over $S$, but it has only finitely many
modulo $\PsAut(X/S)$.

This is related to other finiteness
problems in minimal model theory. We know that a klt pair $(X,\Delta)$
has only finitely many minimal models if $\Delta$ is big
\cite[Corollary 1.1.5]{BCHM}. It follows that a variety of general
type has a finite set of minimal models. A variety
of non-maximal Kodaira dimension can have infinitely many minimal
models \cite[section 6.8]{Reid}, \cite{Kawamata}.
But it is conjectured that every variety $X$
has only finitely many minimal models {\it up to isomorphism}, meaning
that we ignore the birational identification with $X$. Kawamata's
results on Calabi-Yau fiber spaces imply at least that 3-folds of
positive Kodaira dimension have only finitely many minimal models
up to isomorphism \cite[Theorem 4.5]{Kawamata}. If the abundance
conjecture \cite[Corollary 3.12]{KMbook} holds 
(as it does in dimension 3), then every non-uniruled
variety has an Iitaka fibration where the fibers are Calabi-Yau.
The cone conjecture for Calabi-Yau fiber spaces (plus abundance)
implies finiteness of minimal models up to isomorphism for arbitrary varieties.

The cone conjecture is wide open for Calabi-Yau 3-folds,
despite significant results by Oguiso and Peternell \cite{OP},
Szendr\" oi \cite{Szendroi}, Uehara \cite{Uehara}, and Wilson \cite{Wilson}.
Hassett and Tschinkel recently checked the conjecture for a class
of holomorphic symplectic 4-folds \cite{HT}.

\section{Outline of the proof of Theorem \ref{pair}}

The proof of Theorem \ref{pair} gives a good picture of the Calabi-Yau
pairs of dimension 2. We summarize the proof from \cite{Totaro}.
In most cases, if $(X,\Delta)$ is a Calabi-Yau pair, then $X$ turns out
to be rational. It is striking that the most
interesting case of the theorem is proved 
by reducing properties of certain rational surfaces
to the Torelli theorem for K3 surfaces.

Let $(X,\Delta)$ be a klt Calabi-Yau pair of dimension 2.
That is, $K_X+\Delta\equiv 0$, or equivalently
$$-K_X\equiv \Delta,$$
where $\Delta$ is effective. We can reduce to the case where $X$ is smooth
by taking a suitable resolution of $(X,\Delta)$.

If $\Delta=0$, then $X$ is a smooth Calabi-Yau surface,
and the result is known by Sterk, using the Torelli
theorem for K3 surfaces. So assume that $\Delta\neq 0$.
Then $X$ has Kodaira dimension
$$\kappa(X):=\kappa(X,K_X)$$
equal to $-\infty$.
With one easy exception, Nikulin showed
that our assumptions imply that $X$ is {\it rational}
\cite[Lemma 1.4]{AM}.
So assume that $X$ is rational from now on.

We have three main cases for the proof, depending
on whether the Iitaka dimension $\kappa(X,-K_X)$ is 0,
1, or 2. (It is nonnegative because $-K_X\sim_{\R}\Delta\geq 0$.)
By definition, the Iitaka dimension $\kappa(X,L)$ of a line bundle $L$
is $-\infty$ if $h^0(X,mL)=0$ for all positive integers $m$. Otherwise,
$\kappa(X,L)$ is the natural number $r$ such that
there are positive integers $a,b$ and a positive integer $m_0$
with
$am^r\leq h^0(X,mL)\leq bm^r$ for all positive multiples $m$ of $m_0$
\cite[Corollary 2.1.38]{Lazarsfeld}.

\subsection{Case where $\kappa(X,-K_X)=2$}

That is, $-K_X$ is big. In this case, there is an $\R$-divisor $\Gamma$
such that $(X,\Gamma)$ is klt Fano. Therefore $\Nef(X)$ is rational
polyhedral by the cone theorem, and hence the group $\Aut^*(X)$ 
is finite. So Theorem \ref{pair} is true in a simple way.
More generally, for $(X,\Gamma)$ klt Fano of any dimension,
the Cox ring of $X$ is finitely generated, by 
Birkar-Cascini-Hacon-M\textsuperscript{c}Kernan \cite{BCHM}.

This proof illustrates an interesting aspect of working with pairs:
rather than Fano being a different case from Calabi-Yau,
Fano becomes a special case of Calabi-Yau. That is, if
$(X,\Gamma)$ is a klt Fano pair, then there is another effective
$\R$-divisor $\Delta$ with $(X,\Delta)$ a klt Calabi-Yau pair.

\subsection{Case where $\kappa(X,-K_X)=1$}

In this case, some positive multiple of $-K_X$ gives an elliptic fibration
$X\arrow \P^1$, not necessarily minimal. Here
$\Aut^*(X)$ equals the Mordell-Weil group of $X\arrow \P^1$ up to finite index,
and so $\Aut^*(X)\doteq \Z^n$ for some $n$. This generalizes the example
of $\P^2$ blown up at the intersection of two cubic curves.

The $(-1)$-curves in $X$ are multisections of $X\arrow \P^1$ of a certain fixed
degree. One shows that $\Aut(X)$ has only finitely many orbits on the
set of $(-1)$-curves in $X$. This leads to the statement of Theorem
\ref{pair} in terms of cones.

\subsection{Case where $\kappa(X,-K_X)=0$}
\label{kappa0}

This is the hardest case. Here $\Aut^*(X)$ can be a fairly general group acting
on hyperbolic space; in particular, it can be highly nonabelian.

Here $-K_X\equiv \Delta$ where the intersection pairing on the curves
in $\Delta$ is negative definite. We can contract all the curves in $\Delta$,
yielding a singular surface $Y$ with $-K_Y\equiv 0$. Note that $Y$ is klt
and hence has quotient singularities, but it must have worse than ADE
singularities, because it is a singular Calabi-Yau surface that is rational.

Let $I$ be the ``global index'' of $Y$, the least positive integer with
$IK_Y$ Cartier and linearly equivalent to zero. Then
$$Y=M/(\Z/I)$$
for some Calabi-Yau surface $M$ with ADE singularities. The minimal
resolution of $M$ is a smooth Calabi-Yau surface. Using the Torelli
theorem for K3 surfaces, this leads
to the proof of the theorem for $M$ and then for $Y$, by Oguiso
and Sakurai \cite[Corollary 1.9]{OS}.

Finally, we have to go from $Y$ to its resolution of singularities,
the smooth rational surface $X$.
Here $\Nef(X)$ is more complex than $\Nef(Y)$: $X$ typically contains
infinitely many $(-1)$-curves, whereas $Y$ has none (because $K_Y\equiv 0$).
Nonetheless, since we know ``how big'' $\Aut(Y)$ is (up to finite index),
we can show that the group
$$\Aut(X,\Delta)=\Aut(Y)$$
has finitely many orbits on the set of $(-1)$-curves. This leads to the proof
of Theorem \ref{pair} for $(X,\Delta)$. \qed

\section{Example}

Here is an example of a smooth rational surface with a highly nonabelian
(discrete) automorphism group, 
considered by Zhang \cite[Theorem 4.1]{Zhang}, Blache
\cite[Theorem C(b)(2)]{Blache}, and \cite[section 2]{Totaro}.
This is an example of the last case in the proof
of Theorem \ref{pair}, where $\kappa(X,-K_X)=0$. We will also see a singular
rational surface whose nef cone is round, of dimension 4.

Let $X$ be the blow-up of $\P^2$ at the 12 points: $[1,\zeta^i,\zeta^j]$
for $i,j\in \Z/3$, $[1,0,0]$, $[0,1,0]$, $[0,0,1]$. Here $\zeta$ is a cube root
of 1. (This is the dual of the ``Hesse configuration''
\cite[section 4.6]{DolgachevAbstract}.
There are 9 lines $L_1,\ldots,L_9$ through
quadruples of the 12 points in $\P^2$.)

On $\P^2$, we have
$$-K_{\P^2}\equiv 3H\equiv \sum_{i=1}^9 \frac{1}{3}L_i.$$
On the blow-up $X$, we have
$$-K_X\equiv \sum_{i=1}^9 \frac{1}{3} L_i,$$
where $L_1,\ldots,L_9$ are the proper transforms of the 9 lines, which 
are now disjoint and have
self-intersection number $-3$. Thus $(X,\sum_{i=1}^9 (1/3)L_i)$
is a klt Calabi-Yau pair.

Section \ref{kappa0} shows how to analyze $X$:
contract the $9$ $(-3)$-curves $L_i$ on $X$. This gives a rational surface
$Y$ with 9 singular points (of type $\frac{1}{3}(1,1)$) and $\rho(Y)=4$.
We have $-K_Y\equiv 0$, so $Y$ is a klt Calabi-Yau surface which
is rational. We have $3K_Y\sim 0$, and so $Y\cong M/(\Z/3)$ with
$M$ a Calabi-Yau surface with ADE singularities. It turns out that
$M$ is smooth, $M\cong E\times E$ where $E$ is the Fermat cubic curve
$$E\cong \C/\Z[\zeta]\cong \{[x,y,z]\in \P^2: x^3+y^3=z^3\},$$
and $\Z/3$ acts on $E\times E$ as multiplication by $(\zeta,\zeta)$
\cite[section 2]{Totaro}.

Since $E$ has endomorphism ring $\Z[\zeta]$, the group
$GL(2,\Z[\zeta])$ acts on the abelian surface $M=E\times E$.
This passes to an action on
the quotient variety $Y=M/(\Z/3)$
and hence on its minimal resolution $X$ (which is the blow-up
of $\P^2$ at 12 points we started with). Thus the infinite,
highly nonabelian discrete group $GL(2,\Z[\zeta])$ acts on the smooth rational
surface $X$. This is the whole automorphism group of $X$ up to finite groups
\cite[section 2]{Totaro}.

Here $\Nef(Y)=\Nef(M)$ is a round cone in $\R^4$, and so Theorem \ref{pair}
says that $PGL(2,\Z[\zeta])$ acts with finite covolume on hyperbolic 3-space.
In fact, the quotient of hyperbolic 3-space by an index-24 subgroup
of $PGL(2,\Z[\zeta])$
is familiar to topologists as the complement of the figure-eight knot 
\cite[1.4.3, 4.7.1]{MR}.

\small \sc DPMMS, Wilberforce Road,
Cambridge CB3 0WB, England

b.totaro@dpmms.cam.ac.uk

\begin{thebibliography}{99}

\bibitem{Vinberg} D.~Alekseevskij, E.~Vinberg, and A.~Solodovnikov.
Geometry of spaces of constant curvature. {\it Geometry II}, 1--138.
Encylopaedia Math.\ Sci.\ {\bf 29}, Springer (1993).

\bibitem{AM} V.~Alexeev and S.~Mori. Bounding singular surfaces
of general type. {\it Algebra, arithmetic and geometry
with applications }(West Lafayette, 2000), 143--174. Springer (2004).

\bibitem{BCHM} C.~Birkar, P.~Cascini, C.~Hacon, and
J.~M\textsuperscript{c}Kernan.
Existence of minimal models for varieties of log general type.
{\it J.\ Amer.\ Math.\ Soc.\ }{\bf 23 }(2010), 405--468.

\bibitem{Blache} R.~Blache. The structure of l.c.\ surfaces
of Kodaira dimension zero. {\it J.\ Alg.\ Geom.\ }{\bf 4 }(1995),
137--179.

\bibitem{Borcherds} R.~Borcherds. Coxeter groups, Lorentzian
lattices, and K3 surfaces.
{\it Internat.\ Math.\ Res.\ Notices }{\bf 1998}, 1011-1031.

\bibitem{BorelH} A.~Borel and Harish-Chandra. Arithmetic subgroups
of algebraic groups. {\it Ann.\ Math.\ }{\bf 75 }(1962),
485--535.

\bibitem{BorelS} A.~Borel and J.-P.~Serre. Th\'eor\`emes de finitude
en cohomologie galoisienne. {\it Comment.\ Math.\ Helv.\ }{\bf 39 }(1964),
111--164.

\bibitem{BridsonH} M.~Bridson and A.~Haefliger. {\it Metric spaces
of non-positive curvature. }Springer (1999).

\bibitem{CT} A.-M.~Castravet and J.~Tevelev. Hilbert's 14th problem
and Cox rings. {\it Compos.\ Math.\ }{\bf 142 }(2006), 1479--1498.

\bibitem{defernex} T.~de Fernex. The Mori cone of blow-ups
of the plane. arXiv:1001.5243

\bibitem{Debarre} O.~Debarre. {\it Higher-dimensional
algebraic geometry. }Springer (2001).

\bibitem{DolgachevAbstract} I.~Dolgachev. Abstract configurations
in algebraic geometry. {\it The Fano conference, }423--462.
Univ.\ Torino (2004).

\bibitem{DolgachevReflection} I.~Dolgachev. Reflection groups
in algebraic geometry. {\it Bull.\ Amer.\ Math.\ Soc.\ }{\bf 45 }(2008),
1--60.

\bibitem{DZ} I.~Dolgachev and D.-Q.~Zhang. Coble rational surfaces.
{\it Amer.\ J.\ Math.\ }{\bf 123 }(2001), 79--114.

\bibitem{GalM} C.~Galindo and F.~Monserrat. The total coordinate
ring of a smooth projective surface.
{\it J. Alg.\ }{\bf 284 }(2005), 91--101.

\bibitem{Gizatullin} M.~Gizatullin. Rational $G$-surfaces.
{\it Math.\ USSR Izv.\ }{\bf 16 }(1981), 103--134.

\bibitem{Hartshorne} R.~Hartshorne. {\it Algebraic
geometry. }Springer (1977).

\bibitem{HT} B.~Hassett and Y. Tschinkel. Flops on holomorphic
symplectic fourfolds and determinantal cubic hypersurfaces.
{\it J. Inst.\ Math.\ Jussieu }{\bf 9 }(2010), 125--153.

\bibitem{HK} Y.~Hu and S.~Keel. Mori dream spaces and GIT.
  {\it Michigan Math.\ J.\ }{\bf 48 }(2000), 331--348.

\bibitem{IP} V.~Iskovkikh and Yu.~Prokhorov. Fano varieties.
{\it Algebraic geometry V}, 1--247, ed. A.~Parshin and
I.~Shafarevich. Springer (1999).

\bibitem{Kawamata} Y.~Kawamata. On the cone of divisors of
Calabi-Yau fiber spaces. {\it Int.\ J.\ Math.\ }{\bf 8 }(1997), 665--687.

\bibitem{KMM} J.~Koll\'ar, Y.~Miyaoka, and S.~Mori.
Rational connectedness and boundedness of Fano varieties.
{\it J.\ Diff.\ Geom.\ }{\bf 36 }(1992), 765--779.

\bibitem{KMbook} J.~Koll\'ar and S. Mori. {\it Birational geometry of algebraic
varieties. }Cambridge (1998).

\bibitem{Kovacs} S.~Kov\'acs. The cone of curves of a K3 surface.
{\it Math.\ Ann.\ }{\bf 300 }(1994), 681--691.

\bibitem{Lazarsfeld} R.~Lazarsfeld. {\it Positivity in algebraic
geometry}, v.~1. Springer (2004).

\bibitem{MR} C.~Maclachlan and A.~Reid. {\it The arithmetic
of hyperbolic 3-manifolds. }Springer (2003).

\bibitem{Mazur} B.~Mazur. The passage from local to global
in number theory. {\it Bull.\ AMS }{\bf 29 }(1993), 14--50.

\bibitem{Mukaifin} S.~Mukai.  Finite generation of the Nagata
invariant rings in A-D-E cases. RIMS preprint 1502 (2005).

\bibitem{MumfordAV} D.~Mumford. {\it Abelian varieties. }Tata
(1970).

\bibitem{Namikawa} Y.~Namikawa. Periods of Enriques surfaces.
{\it Math.\ Ann.\ }{\bf 270 }(1985), 201--222.

\bibitem{Nikulinint} V.~Nikulin. Integral symmetric bilinear forms
and some of their geometric applications.
{\it Math.\ USSR Izv.\ }{\bf 14 }(1979), 103--167 (1980).

\bibitem{OP} K.~Oguiso and T.~Peternell. Calabi-Yau threefolds
with positive second Chern class.
{\it Comm.\ Anal.\ Geom. }{\bf 6 }(1998), 153--172.

\bibitem{OS} K.~Oguiso and J.~Sakurai. Calabi-Yau threefolds
of quotient type. {\it Asian J.\ Math.\ }{\bf 5 }(2001), 43--77.

\bibitem{Prasad} G.~Prasad. Discrete subgroups isomorphic to lattices
in Lie groups. {\it Amer.\ J.\ Math.\ }{\bf 98 }(1976), 853--863.

\bibitem{Reid} M.~Reid. Minimal models of canonical 3-folds.
{\it Algebraic varieties and analytic varieties }(Tokyo, 1981),
131--180. North-Holland (1983).

\bibitem{Serre} J.-P.~Serre. Arithmetic groups. {\it Homological
group theory }(Durham, 1977), 105--136, Cambridge (1979);
{\it Oeuvres}, v.~3, Springer (1986), 503--534.

\bibitem{Shioda} T.~Shioda. On elliptic modular surfaces.
{\it J.\ Math.\ Soc.\ Japan }{\bf 24 }(1972), 20--59.

\bibitem{Sterk} H.~Sterk. Finiteness results for algebraic
K3 surfaces. {\it Math.\ Z.\ }{\bf 189 }(1985), 507--513.

\bibitem{Szendroi} B.~Szendr\" oi. Some finiteness results
for Calabi-Yau threefolds. 
{\it J.\ London Math.\ Soc.\ }{\bf 60 }(1999), 689--699.

\bibitem{TVV} D.~Testa, A.~V\'arilly-Alvarado, and M.~Velasco.
Big rational surfaces. arXiv:0901.1094

\bibitem{Tits} J.~Tits. Classification of algebraic semisimple
groups. {\it Algebraic groups and discontinuous subgroups }(Boulder,
1965), 33--62. Amer.\ Math.\ Soc. (1966).

\bibitem{Totaro} B.~Totaro. The cone conjecture
for Calabi-Yau pairs in dimension two.
{\it Duke Math.\ J.\ }{\bf 154 }(2010), 241--263.

\bibitem{Uehara} H.~Uehara. Calabi-Yau threefolds with infinitely
many divisorial contractions.
{\it J.\ Math.\ Kyoto Univ.\ }{\bf 44 }(2004), 99--118.

\bibitem{Wilson} P.M.H.~Wilson. Minimal models of Calabi-Yau
threefolds. {\it Classification of algebraic varieties }(L'Aquila,
1992), 403--410.

\bibitem{Yau} S.-T. Yau. On the Ricci curvature of a compact
K\"ahler manifolds and the complex Monge-Amp\`ere equation. I.
{\it Comm.\ Pure Appl.\ Math.\ }{\bf 31 }(1978), 339--411.

\bibitem{Zhang} D.-Q.~Zhang. Logarithmic Enriques surfaces, I, II.
{\it J.\ Math.\ Kyoto Univ.\ }{\bf 31 }(1991), 419--466;
{\bf 33 }(1993), 357--397.

\end{thebibliography}
\end{document}